\documentclass{article}
\usepackage[utf8]{inputenc}
\usepackage{amsmath}
\usepackage{amsmath,amssymb}
\usepackage{euscript}

\usepackage[utf8]{inputenc}
\usepackage[english]{babel}

\usepackage{blindtext}
\usepackage{amssymb}

\usepackage{amsthm}

\DeclareMathOperator{\sign}{sgn}

\newtheorem{theorem}{Theorem}
\newtheorem{corollary}{Corollary}

\newtheorem{proposition}{Proposition}

\theoremstyle{remark}

\theoremstyle{definition}
\newtheorem{definition}{Definition}
\newtheorem{example}{Example}
\newtheorem{case}{Case}
\usepackage{amsthm}

\numberwithin{equation}{section}
\numberwithin{corollary}{section}
\numberwithin{definition}{section}
\numberwithin{theorem}{section}

\usepackage{graphicx}
\usepackage{yhmath}
\usepackage{mathdots}

\usepackage{flexisym}

\title{\textsc{Ultra-recursive sequences}}
\author{\textsc{Óscar Andrés Ram. Ramírez}}
\date{February 2019}

\newcommand{\Addresses}{{
  \footnotesize
\textsc{Centro de Nanociencias y Nanotecnolog\'ia, Universidad Nacional Aut\'onoma de M\'exico, Apdo. Postal 14, 2280 Ensenada B. C., M\'exico}\par\nopagebreak
  \textit{E-mail address},  \texttt{nl\_rami15@cnyn.unam.mx}
}}

\begin{document}

\maketitle
\begin{abstract}
We study a new type of sequences whose elements are defined in terms of the position, sign and magnitude of another element of the sequence. The name \textit{ultra-recursive} comes from the fact that these sequences possess terms that are generated adding either the previous (as in traditional recurrences formulas) or the next terms. They are also \textit{self-referential} because the rules to generate every member alludes to some value of the sequence.
\end{abstract}

\maketitle

\section{Introduction}

Our ultimate purpose is to explore transformations that map a sequence $(\Lambda_k)$ to $(\Lambda\textprime_k)$ so that every term $\Lambda\textprime_{p+1}$ is the sum of previous or later terms of $\Lambda_p$. Therefore, it is convenient to use sequences whose domain are the integers, namely bi-infinite sequences or doubly infinite sequences, who do not have an initial nor a last element:
\begin{equation*}
    (\Lambda_k)^{\infty}_{k=-\infty}=(\ldots,\Lambda_{-3},\Lambda_{-2},\Lambda_{-1},\Lambda_0,\Lambda_1,\Lambda_2,\Lambda_{3},\ldots)
\end{equation*}

Before we define any ultra-recursive sequence, it is necessary to remind some properties of the sequences defined by traditional recurrence relations. Let's consider the following transformation $\textsc{G}$:
\begin{equation}
    \textsc{G}\circ(\Lambda_k)\equiv (\Lambda\textprime_k) \colon \ \Lambda\textprime_q=P*\Lambda_{q-1}-Q*\Lambda_{q-2} 
\end{equation}
where $P$ and $Q$ are integers.

For a given sequence of sequences $((\Lambda_{j,k})^{\infty}_{k=-\infty})^{\infty}_{j=-\infty}\equiv\boldsymbol{\Lambda}$, we can define the transformation $\textsc{\textbf{G}}_l$ that maps any sequence $(\Lambda_{j,k})$ to the element $(\Upsilon_{j+l,k})$ of $\boldsymbol{\Upsilon}$ as follows
\begin{equation}
    \textsc{\textbf{G}}_l\circ\boldsymbol{\Lambda}\equiv\boldsymbol{\Upsilon}\colon\ (\Upsilon_{j+l,k})=\textsc{G}\circ(\Lambda_{j,k})
\end{equation}

If $\boldsymbol{\Upsilon}=\boldsymbol{\Lambda}$, it means that $\textsc{\textbf{G}}_l$ maps $\boldsymbol{\Lambda}$ to itself. Moreover, if $r$ is the least number that satisfy $\textsc{\textbf{G}}^r_l\circ\boldsymbol{\Lambda}=\boldsymbol{\Lambda}$, we know the sequence remains invariant after applying (any multiple of) $r$ times the transformation; in that case, we say that $\boldsymbol{\Lambda}$ is an eigen-sequence of the transformation $\textsc{\textbf{G}}^r_l$. There are several ways to find such sequences of sequences, some of them are rather complicated. For $l=0$, $\textsc{\textbf{G}}^r_0\circ\boldsymbol{\Lambda}=\boldsymbol{\Lambda}$ implies $\textsc{G}^r\circ(\Lambda_{j,k})=(\Lambda_{j,k})$, and this can take place if and only if each element satisfy the relation\footnote{In equation (1.3), $\Lambda_{j,q}$ and $\Lambda_{j,q-r-i}$, which are elements of any $j$-th sequence in $\boldsymbol{\Lambda}$ where substituted just by $\Lambda_{q}$ and $\Lambda_{q-r-i}$ for a clearer exposition of the main idea.}
\begin{equation}
    \Lambda_{q}=\sum^{r}_{i=0}(-1)^i\binom{r}{i}P^{r-i}Q^i*\Lambda_{q-r-i}
\end{equation}
For the sequences who do ``have a beginning'', this equation admits $2r$ initial values, the next elements are a combination of the previous ones. But even those sequences can be extended ``to the left'', meaning that it is possible to calculate predecessors of the initial values and make it a bi-infinite sequence so that (1.3) is still valid throughout the sequence.

Equation (1.3) is a linear homogeneous recurrence relation of degree $r+1$ with constant coefficients. It allow us to find eigen-sequences of the transformation $\textsc{G}^r$ by doing some calculations;\footnote{It can also be proved that given $({\Lambda}_k)$ that remains invariant under $\textsc{G}^r$, and $(\dot{\Lambda}_k)$ that remains invariant under $\textsc{G}^s$, then its sum $(\ddot{\Lambda}_k)\equiv (\Lambda_k)+(\dot{\Lambda}_k)$ has the property $\textsc{G}^t\circ(\ddot{\Lambda}_k)=(\ddot{\Lambda}_k)$, where $t$ is the least common multiple of $r$ and $s$ and in general $\textsc{G}^n\circ(\ddot{\Lambda}_k)\neq(\ddot{\Lambda}_k)$ for $n<t$.} in this paper, the purpose of its presence is merely to expose how easy and mechanical it is to find eigen-sequences of the transformation (1.1) and to show that when $r=1$, (1.3) is exactly the recurrence relation of the Lucas sequence. 

It's is well known that the Lucas sequence has as complementary instances the Fibonacci sequence $(F_n)^\infty_{n=0}=(0,1,1,2,3,5,8,13,\ldots)$ and the Lucas numbers $(L_n)=(2,1,3,4,7,11,18,\ldots)$, \cite{koshy} both satisfying the recurrence relation
\begin{equation}
    \Lambda_p=\Lambda_{p-1}+\Lambda_{p-2}
\end{equation}
That is (1.3) when $r=P=-Q=1$. Any closed form solution for this kind of sequences is expressed in terms of the numbers $\varphi=\frac{1}{2}(1+\sqrt{5})$ and $\psi=1-\varphi=-\varphi^{-1}$, who are the roots of the second order equation $x^2=x+1$, both having the property
\begin{equation}
    \phi^{n+2}=\phi^{n+1}+ \phi^{n}
\end{equation} 

The closed-form expression for the Fibonacci sequence is
\begin{equation}
    F_n=\frac{1}{\sqrt{5}}(\varphi^n-\psi^n)
\end{equation}

For Lucas numbers, it is
\begin{equation}
    L_n=\varphi^n+\psi^n
\end{equation} 

Any sequence defined by (1.4) has an infinite number (all equivalent) of closed-form expressions
\begin{equation}
    \Lambda_k=\beta_e\varphi^{k-e}+\gamma_e\psi^{k-e}
\end{equation}
where $\beta_e$ and $\gamma_e$ are the constants that, when substituted, generate the two consecutive values $\Lambda_e$ and $\Lambda_{e+1}$. Thus
\begin{equation}
    \begin{pmatrix}
    1 & 1\\\phi & \psi
    \end{pmatrix}
    \begin{pmatrix}
    \beta_e\\\gamma_e
    \end{pmatrix}=
    \begin{pmatrix}
    \Lambda_e\\\Lambda_{e+1}
    \end{pmatrix}\\
    \iff
    \begin{pmatrix}
    \beta_e\\\gamma_e
    \end{pmatrix}=\frac{1}{\sqrt{5}}\begin{pmatrix}
    A_{e}\varphi^{-1}+A_{e+1}\\A_e\varphi-A_{e+1}
    \end{pmatrix}
\end{equation}
These results will be useful in the following sections.


\subsection{Sequences satisfying strange recurrence relations}
In his famous book \textit{G\"odel, Escher, Bach: an Eternal Golden Braid} \cite{hofstadter}, Douglas R. Hofstadter introduced the sequence $(Q_n)^\infty_{n=1}$ with the initial values $Q_1=Q_2=1$ and the following relation:
\begin{equation*}
    Q_n=Q_{n-Q_{n-1}}+Q_{n-Q_{n-2}}\quad \text{for} \ n>2
\end{equation*}
which he called a ``strange'' recursion because each value depends on two previous elements, but not the immediately previous two values, like in the Fibonacci sequence. The first elements of $(Q_n)$ are
\begin{equation*}
    (Q_n)=(1,1,2,3,3,4,5,5,6,6,6,8,8,8,10,9,10,\ldots)
\end{equation*}

Hofstadter's $Q$ sequence has been studied both analytically and numerically, it is studied in \cite{pinn1} but little is known about it and it's chaotic behaviour; it's thought that there are an infinite number numbers that does not appear on $(Q_n)$! 

Inspired by $(Q_n)$, other recursions were introduced. For example, the Conway sequence:
\begin{equation*}
    C_n=C_{C_{n-1}}+C_{n-C_{n-1}} \quad \text{for} \ n>2
\end{equation*}
with the initial values $C_1=C_2=1$. The first elements of $C_n$ are
\begin{equation*}
    (C_n)=(1,1,2,2,3,4,4,4,5,6,7,7,8,8,8,8,9,\ldots)
\end{equation*}

\subsection{The general ultra-recursive transformation}

In the next two sections, we'll study some particular cases of this type of transformation of sequences of numbers
\begin{equation}
    \textsc{H}(F_1,F_2,\ldots,F_6)\circ(u_k)\equiv(\nu_k)\colon \ \nu_{p+1}=\sum^{F_2-1}_{i=F_1}[F_3* u_{p F_4-i F_5(\sign u_p)}+F_6]
\end{equation}
where $(u_k)$ is a bi-infinite sequence, $p$ is the position of any of it's elements and $F_1,...,F_6$ are any kind of functions.\footnote{In (1.10), the upper bound of the summation was chosen by simplicity, so that $F_2-F_1$ is the number of summands.} The relevant property of this transformation is that $\nu_{p+1}$ can be a sum of the successors of $u_p$.

\section{Ultra-recursive sequences}

We are now going to study the eigen-sequences of one of the simplest non-trivial cases of (1.10): the transformation $\textsc{H}(0,|u_p|,1,1,\sign{u_p},1)$. From now on, we will refer to it just as $\textsc{O}$:
\begin{equation}
    \textsc{O}\circ(u_k)\equiv(\nu_k)\colon \ \nu_{p+1}=\sum^{|u_p|-1}_{i=0}[u_{p-i\sign u_p}+1]
\end{equation}

A sequence that remains invariant under this transformation must satisfy a strange recurrence formula, we will refer to it as ultra-recursive sequence.

\begin{definition}
We call ultra-recursive sequence $(u_k)_{k\in\mathbb{Z}}$ to any sequence whose elements satisfy the  formula
\begin{equation}
u_{p+1}=\sum^{|u_p|-1}_{i=0}[u_{p-i\sign u_p}+1]
\end{equation}
By definition, $(u_k)$ is an eigen-sequence of $\textsc{O}$. This is: $\textsc{O}\circ(u_k)=(u_k)$.
\end{definition}

There are two interesting interpretations of Definition 2.1, both arise from the equivalent equation:
\begin{equation}
    u_{p+1}=|u_p|+\sum^{|u_p|-1}_{i=0}u_{p-i\sign u_p}
\end{equation}

\textbf{Interpretation 1.} In $(u_k)$, every term generates its successor according to the following rules:  $u_{p+1}$ is equal to $|u_p|$ plus another $|u_p|$ elements: $u_p$ and the previous $[u_p-1]$ elements (if $u_p$ is \textbf{positive}) or $u_p$ and the next $[|u_p|-1]$ elements (if $u_p$ is \textbf{negative}). If $u_p=0$, no other value would be added.

\textbf{Interpretation 2.} $(u_k)$ is a self-descriptive sequence since any two consecutive members give information about the sum of a subset of the sequence: $u_{p+1}-|u_p|=\sum u_{p \pm i}$. Specifically: if $u_p$ is positive, $(u_{p+1}-u_p)$ will always be the sum of the $u_p$ predecessors of $u_{p+1}$; if $u_p$ is negative, $(u_{p+1}+u_p)$ will be the sum of the $|u_p|$ successors of $u_{p-1}$.

What kind of collection of numbers satisfy an equation of this nature? Note that we didn't define any initial terms and we won't do it in the near future because, roughly speaking, it is not possible to give any value we desire to a set of terms. If we say $u_0=3$ and $u_1=5$, for example, we're saying the first generates the second by the sum $|3|+(3+u_{-1}+u_{-2})$, but the last two elements are not defined yet, all we know is $u_{-1}+u_{-2}=-1$ and $u_{-1}$ must generate $u_0=3$ by (2.2) and it is not evident that there exist two values that can satisfy all those requirements.\\
 
Till this moment, it is unclear what combination of elements can be an ultra-recursive sequence. Let's see what happens if the value $1$ exists in the position $p$.

\begin{corollary}
According to (2.3), if there exists an ultra-recursive sequence $(u_k)$ with $u_p=1$, it's successors would be
    \begin{equation*}
    \begin{split}
        &u_{p+1}=1+\sum^{0}_{i=0}u_{p-i}=1+u_p=1+1=2\\
        &u_{p+2}=2+\sum^1_{i=0}u_{p+1-i}=2+u_{p+1}+u_p=2+2+1=5
        \end{split}
    \end{equation*}
    Meaning $u_p=1 \implies (u_k)=(\ldots,u_{p-2},u_{p-1}, \ \boldsymbol{1} \ , \ \boldsymbol{2} \ , \ \boldsymbol{5} \ ,u_{p+3},u_{p+4}\ldots)$.
\end{corollary}

We didn't give a universal value to $u_{p+3}$ because it is dependent on its 5 predecessors and $u_{p-1}$ and $u_{p-2}$ are undefined. In general, it is hard to propose \textit{manually} possible values for $(u_k)$, but we can find more about the restrictions imposed by Definition 2.1. 

\subsection{Cases of Definition 2.1}
Given a presumably ultra-recursive sequence $(u_k)$ with an arbitrary element $u_p$
    \begin{equation*}
        (u_k)=(\ldots, u_{p-3}, u_{p-2}, u_{p-1}, \boldsymbol{u_p}, u_{p+1}, u_{p+2}, u_{p+3} \ldots)
    \end{equation*}
there are three possible cases: $u_p>0, \ u_p<0$ \ and \ $u_p=0$.
\begin{case}
    When $u_p$ is positive, 
    \begin{equation*}
        u_p>0 \iff |u_p|=u_p \iff \sign u_p=1   
    \end{equation*}
    and (2.3) will be equivalent to:
    \begin{equation}
        u_{p+1}=u_p+\sum^{u_p-1}_{i=0}u_{p-i}=2u_p+\sum^{u_p-1}_{i=1}u_{p-i}
    \end{equation}
    and also equivalent to the following equations
    \begin{equation}
        u_{p+1}=2u_p+\sum^{p-1}_{i=p+1-u_p}u_i
    \end{equation}
    \begin{equation}
        u_{p+1}=2+\sum^{p-1}_{i=p+1-u_p}(u_i+2)
    \end{equation}
\end{case}
\begin{case}
    If $u_p$ is negative,
    \begin{equation*}
        u_p<0 \iff |u_p|=-u_p \iff \sign{u_p}=-1
    \end{equation*}
    and by (2.3): 
    \begin{equation*}
        u_{p+1}=-u_p+\sum^{-u_p-1}_{i=0}u_{p+i}=-u_p+u_p+\sum^{-u_p-1}_{i=1}u_{p+i}=\sum^{p-1-u_p}_{i=p+1}u_i
    \end{equation*}
    Extracting the first element from the sum (when $i=p+1$), gives
    \begin{equation}
        u_{p+1}=u_{p+1}+\sum^{p-1-u_p}_{i=p+2}u_{i}
    \end{equation}
    And this can only be true if the sum is equal to zero. We have found the first important conclusion about $(u_k)$.
    
    \begin{corollary}
        In a u-sequence $(u_k)$, if the element $u_p$ is negative, then
        \begin{equation}
            \sum^{p-1-u_p}_{i=p+2}u_i=0
        \end{equation}
    \end{corollary}
\end{case}
As we shall see later, this Corollary also means that given any sequence $(\upsilon_k)$, if there are consecutive elements  whose sum is zero, that is if $\sum^\beta_{i=\alpha}\upsilon_{i}=0$ and if $\upsilon_{\alpha-2}=\alpha-\beta-3$, then the element $\upsilon_{\alpha-1}$ will remain invariant under the transformation $\textsc{O}$ ($\upsilon^\prime_{\alpha-1}=\upsilon_{\alpha-1}$).\\

\begin{corollary}
    From (2.2) and by the definition of summation, it's clear that 
    \begin{equation*}
        u_p=0\implies u_{p+1}=0
    \end{equation*}
    consecutively
    \begin{equation}
        u_p=0\implies u_q=0 \ \ \forall \ \ q>p
    \end{equation}
\end{corollary}
But this is evident. According to Interpretation 1, if every term produces the next by adding as many numbers as it's value, zero must produce another zero or a bored infinite sequence of zeros to its right. But why didn't we use the logical operator for \textbf{bi-implication} in Corollary 2.3? What other elements can generate a zero, apart from zero?
It's time to ask ourselves seriously: at what point it is a lose of time to study an equation as arbitrary as (2.2)? The following example will partially solve this concerns and will lead us to the discovery of an interesting number that will allow us to create, manipulate or propose ultra-recursive sequences with unexpected properties. 

\begin{example}
    Corollary 2.2 predicts the existence of a member of $(u_k)$ equal to zero. The summation has only one summand, when its lower and upper bound are equal:
    \begin{equation*}
        \sum^{p-1-u_p}_{i=p+2}u_i=\sum^{\eta}_{i=\eta}u_i=u_\eta \iff p+2=p-1-u_p \iff u_p=-3
    \end{equation*}
    There is only one summand when $u_p=-3$ and it has to be $0$, according to Corollary 2.2: $u_p=-3 \implies u_{p+2}=0$. And according to Corollary 2.3: $u_p=-3 \implies u_q=0 \ \ \forall \ \ q>p+1$
    
    Although $-3$ implies an infinite sequence of zeros two places at its right, it does not produce them; what produces is any number that has the audacity to produce a zero. Since $ \sum^{p+3+m}_{i=p+3}u_i=0$ for an arbitrary $m>0$, the Corollary 2.2 establishes that in this case, $u_{p+1}$ can have any negative value $-m$: this means that $u_p=-3$ can generate any negative number.\\
    
    Finally, let's see that both $u_{p+1}=-1$ and $u_{p+1}=-2$ can produce the element $u_{p+2}=0$ without the necessity a large amount of elements equal to zero. 
    \begin{corollary}
    $u_p=-1\implies u_{p+1}=0$ as a consequence of (2.2):
    \begin{equation*}
        u_p=-1 \implies u_{p+1}=\sum^0_{i=0}(u_{p-i}+1)=-1+1=0
    \end{equation*}
    \end{corollary}
    \begin{corollary}
    $u_p=-2\implies u_{p+1}=u_{p+1}$ \textsc{(?)}. According to (2.3):
    \begin{equation*}
    \begin{split}
        u_p=-2 \implies u_{p+1}&=|-2|+\sum^{1}_{i=0}u_{p+i}\\
        &=|-2|+(-2)+u_{p+1}\\
        &=u_{p+1}
    \end{split}
    \end{equation*}    
    \end{corollary}
    This means that $u_p=-2$ generates $u_{p+1}$ without any restriction! A $-2$ in $(u_k)$ allow us to generate any value without compromising any other element in the sequence (although its mere existence can influence the magnitude of other elements).\\
    
    Now, we have an element $u_{p-1}=-2$ that can generate $u_p=-3$, which can generate any negative value $u_{p+1}=-m$ which generate $u_{p+2}=0$ and so on. 
    \begin{corollary}
    There exists an ultra-recursive sequence $(u_k)$ with $u_z=-2$ for $z<0$, $u_0=-3$, $u_1=-m$ with $m \in \mathbb{N}$ and $u_n=0$ for $n>1$:
    \begin{equation*}
        (u_k)=(\ldots,-2,-2,-2,-3,-m,0,0,0,\ldots)
    \end{equation*}
    \end{corollary}
\end{example}

\subsection{Some ultra-recursive sequences}
The following affirmations are easily demonstrated with the results obtained throughout this section.  
\begin{itemize}
    \item $\exists \ (u_k) \colon u_p=0 \ \ \forall \ \ p \in \mathbb{Z}$
    \item $\exists \ (u_k) \colon u_p=-2 \ \ \forall \ \ p \in \mathbb{Z}$
    \item $\exists \ (u_k) \colon u_z = -2 \ \ \forall \ \ z<p \ \ \wedge \ \ u_{n} = 0 \ \ \forall \ \ p+1 \leq n$
    \begin{equation*}
        (u_k)=(\ldots,-2,-2,-2,u_p,u_{p+1},\ldots,0,0,0,\ldots)
    \end{equation*}
    There are countless combinations of values for $u_p$ and its closest successors.
    \item $\exists \ (u_k) \colon u_z=-2 \ \ \forall \ \ z<0 \ \ \wedge \ \ u_0=m \colon m \in \mathbb{Z}^+$. We'll define $\boldsymbol{\Pi}$ to the sequence containing all the possible sequences $(\pi_{m,k})$: alluding the number $m$ in position $0$ (i.e. $\pi_{m,0}=m$).\\
    The elements $\pi_{m,n}$ for $n\geq0$ will be studied in the next section.
\end{itemize}

\section{$\boldsymbol{\Pi}$ sequence}
We've generated some eigen-sequences of the transformation $\textsc{O}$, but we aren't calculating values with an explicit formula, instead we are \textbf{discovering} values that satisfy our definitions. 

Nevertheless, there are ultra-recursive sequences that are partially periodic. In the past section we found sequences $(u_k)$ with infinite terms equal to $-2$ through the left. 
\begin{definition}
The sequence of sequences $\boldsymbol{\Pi}$ has elements $(\pi_{m,k})$ who are eigen-sequences of the transformation $\textsc{O}$.
\begin{equation*}
    \boldsymbol{\Pi} \equiv ((\pi_{m,k})_{k\in\mathbb{Z}})_{m\in\mathbb{Z}^+} \colon \quad \pi_{m,z}=-2 \ \ \forall \ \ z<0 \ \ \ \textit{and} \ \ \ \pi_{m,0}=m
\end{equation*}
\end{definition}

We can generate the terms $\pi_{m,n}$ for $n>0$ iteratively using equation (2.2) or (2.3). The correspondent matrix for $\boldsymbol{\Pi}$ is 
\setcounter{MaxMatrixCols}{20}
\begin{equation*}
    \begin{pmatrix}
    \ldots  & -2 & -2 & -2 & 1 & 2 & 5 & 9 & 16 & 27 & 45 & 74 & \ldots\\
    \ldots  & -2 & -2 & -2 & 2 & 2 & 6 & 10 & 18 & 30 & 50 & 82 & \ldots\\
    \ldots  & -2 & -2 & -2 & 3 & 2 & 7 & 11 & 20 & 33 & 55 & 90 & \ldots\\
    \ldots  & -2 & -2 & -2 & 4 & 2 & 8 & 12 & 22 & 36 & 60 & 98 & \ldots\\
    \ldots  & -2 & -2 & -2 & 5 & 2 & 9 & 13 & 24 & 39 & 65 & 106 & \ldots\\
    \ldots  & -2 & -2 & -2 & 6 & 2 & 10 & 14 & 26 & 42 & 70 & 114 & \ldots\\
    \ldots  & -2 & -2 & -2 & 7 & 2 & 11 & 15 & 28 & 45 & 75 & 122 & \ldots\\
    \ldots  & -2 & -2 & -2 & 8 & 2 & 12 & 16 & 30 & 48 & 80 & 130 & \ldots\\
    \reflectbox{$\ddots$} &
    \vdots & \vdots &\vdots & \vdots & \vdots & \vdots & \vdots & \vdots & \vdots &\vdots & \vdots & \ddots
    \end{pmatrix}
\end{equation*}

What patterns or properties can you find just by looking?

\begin{definition}
The summation of the first $n$ successors of $\pi_{m,-1}$ and the summation of the first $n$ predecessors of $\pi_{m,0}$
\begin{equation*}
     S^m_n\equiv\sum^{n-1}_{i=0}\pi_{m,i}, \ \ \ R^m_n\equiv \sum^{-1}_{i=-n}\pi_{m,i}
 \end{equation*}
\end{definition}
\begin{corollary}
From Definition 3.1, we know that $R^m_n=n(-2)$.
\end{corollary}
\begin{theorem}
    For every $\pi_{m,n+1}$ with $n\geq0$
    \begin{equation}
        \pi_{m,n+1}=2+\sum^{n-1}_{i=0}(\pi_{m,i}+2)=S^m_n+2n+2
    \end{equation}
\end{theorem}
\begin{proof}
    If $\pi_{m,n}>n>0$, we can use (2.6) as follows:
    \begin{equation*}
        \pi_{m,n+1}=2+\sum^{n-1}_{i=n+1-\pi_{m,n}}(\pi_{m,i}+2) =  2+\sum^{n-1}_{i=0}(\pi_{m,i}+2)+\sum^{-1}_{i=n+1-\pi_{m,n}}(\pi_{m,i}+2)
    \end{equation*}
    by Definition 3.1, we know $\pi_{m,z}=-2$ for $z<0$. Therefore:
    \begin{equation*}
    \begin{split}
        \pi_{m,n+1}&=2+\sum^{n-1}_{i=0}(\pi_{m,i}+2)+\sum^{-1}_{i=n+1-\pi_{m,n}}(-2+2)\\
        &=2+\sum^{n-1}_{i=0}(\pi_{m,i}+2)=S^m_n+2n+2\\
    \end{split}
    \end{equation*}
    If $S^m_n>\pi_{m,n}$, the last equation gives $\pi_{m,n+1}$ a positive value and it can be used for $\pi_{m,n+2}$ (since $\pi_{m,n+1}>n+1>0$ was the first condition of the proof) and for induction it can be used for all next elements. Since $S^m_0$ is exactly zero and $\pi_{m,0}$ is always greater than zero, this theorem is valid for $\pi_{m,n+1}$ with $n\geq 0$.
\end{proof}
\begin{corollary}
    In $(\pi_{m,n})$, all the successors of $\pi_{m,0}$ are positives and the first ones are
    \begin{equation*}
        \pi_{m,1}=2, \ \pi_{m,2}=m+4, \ \pi_{m,3}=m+8, \ \pi_{m,4}=2m+14, \ldots
    \end{equation*}
\end{corollary}
Note that $\pi_{m,1}=2$ is the only constant value or the only element independent of $m$ in $(\pi_{m,k})$. 

It's natural to wonder if there is a more organic relation between one element and the previous ones.

\begin{theorem}
Any element of $(\pi_{m,k})$ except from $\pi_{m,0}$ and $\pi_{m,1}$ is equal to the sum of the two previous elements plus two.
\begin{equation}
    \pi_{m,p}=\pi_{m,p-1}+\pi_{m,p-2}+2 \ \ \forall p \neq 0, 1
\end{equation}
\end{theorem}
\begin{proof}
For $z<0$, from Definition 3.1
\begin{equation*}
\pi_{m,z+1}=-2=-2-2+2=\pi_{m,z-1}+\pi_{m,z-2}+2
\end{equation*}
For $n+1\geq2$, we can use Theorem 3.1 as follows
\begin{equation*}
\begin{split}
    \pi_{m,n+1}&=S^m_n+2n+2=(S^m_{n-1}+\pi_{m,n-1})+2n+2\\
    &=(S^m_{n-1}+2(n-1)+2)+\pi_{m,n-1}+2\\
    &=\pi_{m,n}+\pi_{m,n-1}+2
\end{split}
\end{equation*}
\end{proof}

Theorem 3.2 tell us that an infinite subset of $(\pi_{m,k})$ (actually the totality of it minus two values) satisfies not only (2.2) but also a recurrence relation in the traditional fashion. The next natural step is to find the closed form expression for any sequence in $\boldsymbol{\Pi}$.

\begin{theorem}
For $n\geq0$, the function that generates $\pi_{m,n}$ is:
\begin{equation}
    \pi_{m,n}=B_m\varphi^n+C_m\psi^n-2
\end{equation}
or equivalently with $b_m=\sqrt{5}B_m$ and $c_m=\sqrt{5}C_m$ 
\begin{equation*}
    \pi_{m,n}=\frac{1}{\sqrt{5}}(b_m\varphi^n+c_m\psi^n)-2
\end{equation*}
Where the constants $B_m$ and $C_m$ are defined as follows
\begin{equation*}
\begin{split}
    &B_m \equiv \frac{(m+2)\varphi+2-m}{\sqrt{5}}=\frac{2\varphi^2+m\varphi^{-1}}{\sqrt{5}},\\ 
    &C_m \equiv \frac{(m+2)\varphi-4}{\sqrt{5}}=B_m+\frac{m-6}{\sqrt{5}}
\end{split}
\end{equation*}
\end{theorem}
\begin{proof}
In Complementary Information, we've established that for the recurrence relation $\Lambda_n=\Lambda_{n-1}+ +\Lambda_{n-2}+\varepsilon$, exists the closed form solution
\begin{equation*}
    \Lambda_n=\beta \varphi^n+ \gamma \psi^n - \varepsilon
\end{equation*}
Therefore, we only need to find the constants that generate the two initial terms: $B_m+C_m-2=m, \ B_m\varphi+C_m\psi-2=2$. After solving this by the same method mentioned in the Introduction, we get:
\begin{equation*}
    \begin{pmatrix}
    B_m\\C_m
    \end{pmatrix}=\frac{1}{\sqrt{5}}\begin{pmatrix}
    (m+2)\varphi^{-1}+4\\(m+2)\varphi-4
    \end{pmatrix}=\frac{1}{\sqrt{5}}\begin{pmatrix}
    (m+2)\varphi+2-m\\(m+2)\varphi-4
    \end{pmatrix}
\end{equation*}

\end{proof}
\begin{theorem}
The relation between the $n$th term of any two sequences in $\boldsymbol{\Pi}$: $(\pi_{m,k})$ and $(\pi_{t,k})$.
\begin{equation}
    \pi_{m,n}=\pi_{t,n}+(m-t)F_{n-1}
\end{equation}
Where $F_{n-1}$ is a term from the Fibonacci sequence
\end{theorem}
\begin{proof} Theorem 3.3 and the closed form formula of the Fibonacci sequence imply 
    \begin{equation*}
        \pi_{t,n}+(m-t)F_{n-1}=\frac{1}{\sqrt{5}}\Biggl[\Biggl(b_t+\frac{m-t}{\varphi}\Biggr)\varphi^n+\Biggl(c_t+\frac{t-m}{\psi}\Biggr)\psi^n\Biggr]
    \end{equation*}
    but
    \begin{equation*}
    \begin{split}
        b_t+\frac{m-t}{\varphi}&=[(t+2)\varphi+2-t]\ + \ (m-t)(\varphi-1)=(m+2)\varphi+2-m=b_m\\
        c_t+\frac{t-m}{\psi}&=[(t+2)\varphi-4]\ + \ (t-m)(-\varphi)=(m+2)\varphi-4=c_m
    \end{split}
    \end{equation*}
    Therefore
    \begin{equation*}
    \begin{split}
        \pi_{t,n}+(m-t)F_{n-1}&=\frac{1}{\sqrt{5}}[b_m\varphi^n+c_m\psi^n]\\
        &=\pi_{m,n}
    \end{split}
    \end{equation*}
\end{proof}

We are now able to express any element of $(\pi_{m,n})$ as a function of two consecutive values of a given sequence $(\pi_{t,n})$

\begin{theorem}
The element $\pi_{m,n}$ in function of $\pi_{t,e}$ and $\pi_{t,e+1}$.
\begin{equation}
    \pi_{m,n}=\pi_{t,e} F_{n-e-1}+\pi_{t,e+1} F_{n-e}+2F_{n-e+1}+(m-t)F_{n-1}-2
\end{equation}
\end{theorem}
\begin{proof}

From (1.6), we know there are several ways to express the closed-form solution for any sequence in terms of two consecutive values. In the case of $(\pi_{m,n})$, the equation will be similar to (3.3)
\begin{equation}
    \pi_{t,n}=B_{t,e}\varphi^{n-e}+C_{t,e}\psi^{n-e}-2
\end{equation}
Where $e=0$ is exactly (3.3) with $B_{t,0}=B_t$ and $C_{t,0}=C_t$. For any $e$
\begin{equation*}
    \begin{pmatrix}
    1 & 1\\\varphi & \psi
    \end{pmatrix}
    \begin{pmatrix}
    B_{t,e}\\C_{t,e}
    \end{pmatrix}=
    \begin{pmatrix}
    \pi_{t,e}+2\\\pi_{t,e+1}+2
    \end{pmatrix}\\
\end{equation*}
and the constants $\begin{pmatrix}B_{t,e}\\C_{t,e}\end{pmatrix}$ are
\begin{equation*}
    \frac{1}{\sqrt{5}}\begin{pmatrix}
    \pi_{t,e}\varphi^{-1}+\pi_{t,e+1}+2\varphi\\\pi_{t,e}\varphi-\pi_{t,e+1}-2\psi\end{pmatrix}=\frac{\pi_{t,e}}{\sqrt{5}}\begin{pmatrix}
    \varphi^{-1}\\-\psi^{-1}
    \end{pmatrix}+\frac{\pi_{t,e+1}}{\sqrt{5}}\begin{pmatrix}
    1\\-1
    \end{pmatrix}+\frac{2}{\sqrt{5}}\begin{pmatrix}
    \varphi\\-\psi
    \end{pmatrix}
\end{equation*}

This gives (3.6) a new expression
\begin{equation}
    \pi_{t,n}=\pi_{t,e} F_{n-e-1}+\pi_{t,e+1} F_{n-e}+2F_{n-e+1}-2
\end{equation}
Applying Theorem 3.4 to (3.7) leads directly (3.5).
\end{proof}

\begin{corollary}
    From (3.7), for $e=0$, we have a new way to view equation (3.3)
    \begin{equation}
        \pi_{m,n}=mF_{n-1}+2F_{n+2}-2
    \end{equation}
\end{corollary}

\subsection{The $k$th difference sequence of $(\pi_{m,n})$ and their relations with other sequences}

\begin{definition}
The first difference sequence for $(\pi_{m,n})$:
\begin{equation}
    \Delta (\pi_{m,n})=\Delta^1 (\pi_{m,n})\equiv \pi_{m,n+1}-\pi_{m,n}
\end{equation}

And for $k>1$, the $k$th difference of $(\pi_{m,n})$:
\begin{equation}
    \Delta^k(\pi_{m,n})\equiv\Delta^{k-1}(\pi_{m,n+1})-\Delta^{k-1}(\pi_{m,n})
\end{equation}
\end{definition}
\begin{corollary}
From Corollary 3.3, if $n>0$:
\begin{equation}
   \Delta^k (\pi_{m,n})=mF_{n-1-k}+2F_{n+2-k}
\end{equation}
\end{corollary}
\begin{proposition}
From (3.11), it can be proved that the $k$th difference sequence for $\pi_{m,n}$ satisfies the recurrence relation (1.3)
\begin{equation}
    \Delta^k (\pi_{m,n})=\Delta^k (\pi_{m,n-1})+\Delta^k (\pi_{m,n-2})
\end{equation}

If $\ 1\leq k\leq n-1$, the $n$th term of the $k$th difference sequence is related with an element of $(\pi_{m,n})$ by:
\begin{equation}
    \Delta^k(\pi_{m,n})=\pi_{m,n-k}+2
\end{equation}
\end{proposition}
\begin{proposition}
    Some relations between the $k$th difference sequences and other sequences:
    \begin{equation*}
        \begin{split}
            \Delta^k(\pi_{1,n})&=L_{n+2-k}\\
            \Delta^k(\pi_{2,n})&=4F_{n-k}\\
            \Delta^k(\pi_{6,n})&=4L_{n-k}\\
        \end{split}
    \end{equation*}
\end{proposition}

\subsection{A different u-recursive sequence}
Let's consider the finite sequence $(q_j)^3_{j=-5}$
\begin{equation*}
    (q_j)=(\boldsymbol{-2},-2,-2,-2,-2,1,2,5,9)
\end{equation*}
This sequence is a sub-sequence of $(\pi_{1,k})$, all of its elements (except for $q_{-5}$, the minus two written in bold) are produced by their predecessor by the formula (2.3)
\begin{equation*}
    q_{p+1}=|q_p|+\sum^{|q_p|-1}_{i=0}q_{p-i\sign{q_p}}
\end{equation*}
It's possible to expand the sequence by generating the next element $q_4$ using the previous formula
\begin{equation*}
    q_4=9+\sum^8_{i=0}q_{3-i}=9+9-2+\sum^2_{i=-4}q_i=16
\end{equation*}
The sequence contains exactly the 9 elements the summatory is asking for. Notice that the sum of 7 consecutive elements of the sequence is equal to $0$, this is exactly the requisite for the existence of a negative number in the sequence (Corollary 2.2): $q_p \implies \sum^{p-1-q_p}_{i=p+2}q_i=0$. In this case, $\sum^2_{i=-4}q_i$, so the position is $p=-4-2=-6$ and it's value $q_p=-6-1-2=-9$.
This allow us to expand the sequence, by adding another element on the left.
\begin{equation*}
    (q_j)=(-9,-2,-2,-2,-2,-2,1,2,5,9,16)
\end{equation*}
We are able to add as many $-2$ to the left as we want without violating (2.2), so there's enough elements in the sequence to generate the next term to the right. Furthermore, we can also try to find the similar conditions that allowed $-9$ to exist in the first place: $R_n=S_{n^{\prime}}$.

After repeating this procedure, we get the following ultra-recursive sequence:
\begin{equation*}
        (\ldots,-86_{_{-77}}|-42_{_{-35}}|-20_{_{-15}}|-9_{_{-6}}|-2_{_{-1}},1,2,5,9,16,20,38,42,82,86,\ldots)
\end{equation*}
where the sub-indices allude to the position in which the value is located. The omitted values in between are $-2$.

It is not a coincidence that the magnitude of all the negative values in the sequence also appears as positive. The following theorem explains that and also the fact that $q_{-n}\neq-2 \implies q_{q_{-n}-n}=-2q_{-n}-2$.

\begin{theorem}
For $m>0$, there exists an ultrarecursive sequence $(\pi^*_k)$ with $\pi^*_n=\pi_{m,n}$ for $-m-4\leq n \leq 3$; $\pi^*_{2n}=2(\pi^*_{2n-1}-1)$, $\pi^*_{2n+1}=2(\pi^*_{2n-1}+1)$ for $n>1$; and for $t<0$, if $t=-(\pi^*_{2n-1}-2n+1)$ for some $n>1$, it implies $\pi^*_t=-\pi^*_{2n-1}$, otherwise $\pi^*_t=-2$.
\begin{equation*}
    (...,|-(2m+18)_{_{-(2m+13)}}|-(m+8)_{_{-(m+5)}}|-2_{_{-1}},m,2,m+4,m+8,...)
\end{equation*}

\end{theorem}
\begin{proof}
Let's suppose that the sum of the $\pi^*_n-2$ predecessors of $\pi^*_n$ is zero:
\begin{equation*}
    \sum^{n-1}_{i=-(\pi^*_n-n-2)}\pi^*_i=0\\
\end{equation*}
This allows, according to Corollary 2.2, the existence of the element $-(\pi^*_n-n-2)-(n-1)-3=-\pi^*_n$ in the position $\pi^*_{-(\pi^*_n-n)}$; we also know that $\pi^*_{-(\pi^*_n-n-1)}=-2$ is allowed.
The next element in the sequence will be:
\begin{equation*}
    \pi^*_{n+1}= 2\pi^*_{n}+\sum^{n-1}_{i=n+1-\pi^*_n}\pi^*_i=2\pi^*_n+\pi^*_{-(\pi^*_n-n-1)}+\sum^{n-1}_{i=-(\pi^*_n-n-2)}\pi^*_i=2\pi^*_n-2
\end{equation*}
Since $\pi^*_{-(\pi^*_n-n)}+\pi^*_{-(\pi^*_n-n-1)}+\pi^*_n+\pi^*_{n+1}=2\pi^*_n-4$, it's clear that
\begin{equation*}
    \sum^{n+1}_{i=-(2\pi^*_n-n-2)}\pi^*_i=\sum^{n+1}_{i=-(\pi^*_{n+1}-n)}=0
\end{equation*}
if the $\pi^*_n-2$ predecessors of $\pi^*_{-(\pi^*_n-n)}$ have value $-2$. Therefore, the next element in the sequence:
\begin{equation*}
\begin{split}
    \pi^*_{n+2}&=\pi^*_{n+1}+\sum^{n+1}_{i=n+2-\pi^*_{n+1}}\pi^*_i=\pi^*_{n+1}+\sum^{n+1}_{i=-(\pi^*_{n+1}-n-2)}\pi^*_i\\
    &=\pi^*_{n+1}-\pi^*_{-(\pi^*_{n+1}-n)}-\pi^*_{-(\pi^*_{n+1}-n-1)}+\sum^{n+1}_{i=-(\pi^*_{n+1}-n)}\\
    &=\pi^*_{n+1}+4
\end{split}
\end{equation*}

Now, let us notice that the sum of the $\pi^*_{n+1}+2=\pi^*_{n+2}-2$ predecessors of $\pi^*_{n+2}$ is zero. This is the same condition that started the proof, which allow us to demonstrate by induction that this behavior will remain for the successors of $n$. It is easy to proof that for $n=3$
\begin{equation*}
    \sum^{n-1}_{i=n-(\pi_{m,n}-2)}\pi_{m,i}=0
\end{equation*}
\end{proof}

We have found an infinite number of ultra-rrecursive sequences that are not periodic nor partially periodic.
\begin{definition}
    The sequences of sequences $\boldsymbol{\Pi^*}$ has elements $(\pi^*_{m,k})$ that are eigen-sequences of the transformation $\textsc{O}$.
    \begin{equation*}
    \begin{split}
        \boldsymbol{\Pi^*} \equiv ((\pi^*_{m,k})_{k\in\mathbb{Z}})_{m\in\mathbb{Z}^+} \colon  \quad & \pi^*_{m,n}=\pi_{m,n} \ \ \forall \ \ 0\leq n\leq3 \\
        &\pi^*_{m,2n}=2\pi^*_{2n-1} -2\ \ \forall \ \ n>2 \\
        &\pi^*_{m,2n+1}=\pi^*_{m,2n}+4\ \ \forall \ \ n>1 \\
        & r\neq \pi^*_n-n \implies \pi^*_{-r}=-2\ \ \forall \ \ n>2 \\
        & r=\pi^*_n-n \implies \pi^*_{-r}=-r-n\ \ \forall \ \ n>2 \\
    \end{split}
    \end{equation*}
\end{definition}
\begin{theorem}
   For $(\pi^*_{m,n})$, there is the following solution
   \begin{equation}
       \pi^*_{m,2n}=2^{n-1}(m+10)-6
   \end{equation}
\end{theorem}
\begin{proof}
By Definition 3.4,
\begin{equation*}
\begin{split}
    \pi^*_{m,2n}&=2\pi^*_{m,2(n-1)}+6\\
    \implies \pi^*_{m,2n}&=2(2\pi^*_{m,2(n-2)}+6)+6\\
    \implies \pi^*_{m,2n}&=2(2(2\pi^*_{m,2(n-3)}+6)+6)+6\\
    \implies \pi^*_{m,2n}&=2^{n-2}\pi^*_{m,4}+6\sum^{n-3}_{i=0}{2^i}=2^{n-2}(2\pi^*_{m,3}-2)+6*(2^{n-2}-1)\\
    &=2^{n-1}(\pi^*_{m,3}+2)-6=2^{n-1}(m+10)-6
\end{split}
\end{equation*}
\end{proof}

\section{Periodic ultra-recursive sequences}
We are now going to examine periodic sequences with elements that satisfy the equation (2.3). 

Let's consider the sequence
\begin{equation*}
    (\dot{u}_n)=(-6,-2,-2,-2,6,-2)
\end{equation*}
Here, $\dot{u}_0=-2$, $\dot{u}_1=-6$ and so on. It isn't hard to prove that there exists an ultra-recursive sequence $(u_k)$ with $u_n=\dot{u}_n\colon \ 0 \leq n \leq 5$ and $u_k=u_{k+6n} \ \forall \ n\in\mathbb{Z}$.

\begin{definition}
    For any sequence $(a_k)$ with period $p$, we call the unitary sequence to the sub-sequence containing the element $a_1$ and its $p-1$ successors:
    \begin{equation*}
        (\dot{a}_n)^p_{n=1}\colon \quad \dot{a}_m=a_m \ \ \forall \ \ 1\leq m \leq p
    \end{equation*}
\end{definition}

\begin{theorem}
    For all $m\geq0$, there exists an ultra-recursive sequence $(\tau_k)$ with period $4m+2$ whose unitary sequence contains exactly $m$ elements with value $-4m-2$, $m$ elements with value $4m+2$ and $2m+2$ elements with value $-2$. If $\tau_k\neq-2$, then $\tau_{k-1}=-2$ and $\tau_{k+1}=-2$.
\end{theorem}
\begin{proof}
    Since, $(\tau_k)$ is periodic, the sum of any $4m+2$ consecutive elements is equal to
    \begin{equation*}
    \begin{split}
        S_m &\equiv\sum^{(4m+2)+\alpha-1}_{i=\alpha} \tau_{i}=\sum^{4m+2}_{i=1}\dot{\tau}_{i}\\
        &=m(-4m-2)+m(4m+2)+(2m+2)(-2)=-4m-4
    \end{split}
    \end{equation*}
    If $\tau_k\neq-2$, then $\tau_k=\pm(4m+2)\implies |\pm(4m+2)|=4m+2$ and by (2.3)
    \begin{equation*}
        \tau_{k+1}=(4m+2)+\sum^{4m+1}_{i=0}\tau_{\mp i}=(4m+2)+S_m=-2
    \end{equation*}
    Thus, $\tau_k\neq-2\implies \tau_{k+1}=-2$. By Corollary 2.5, we know that $\tau_k=-2\implies \tau_{k+1}=\tau_{k+1}$, so we know every element generates its successor according to equation (2.2) independently of the position of the elements $\tau_k\neq-2$ in the sequence.
\end{proof}
\begin{definition}
    The sequence of sequences $\boldsymbol{T}$ has elements $(\tau^{\textsc{P},\textsc{N}}_{m,k})$ that have period $4m+2$ and are eigen-sequences of the transformation $\textsc{O}$.
    \begin{equation*}
    \begin{split}
        \boldsymbol{T}\equiv((\tau^{\textsc{p},\textsc{n}}_{m,k})_{k\in\mathbb{Z}})_{m\in\mathbb{Z}^+}\colon  \quad & |\textsc{P}|=|\textsc{N}|=m, \ \textsc{P}\cap \textsc{N}=\varnothing, \\
        & \textsc{Q}\equiv\textsc{P}\cup \textsc{N}\implies \textsc{Q} \subset\{1,2,...,4m+2\}\\
        &\forall q_i,q_j\in\textsc{Q}\colon \ q_i\neq q_j+1 \ \bmod{(4m+2)}\\
        &\forall q\notin \textsc{Q}, \dot{\tau}^{\textsc{p},\textsc{n}}_{m,q}=-2\\
        &\forall p\in\textsc{P}, \dot{\tau}^{\textsc{p},\textsc{n}}_{m,p}=4m+2\\  
        &\forall n\in\textsc{N},\dot{\tau}^{\textsc{p},\textsc{n}}_{m,n}=-(4m+2)\\
    \end{split}
    \end{equation*}
    There are several sets \textsc{P} and \textsc{N} that satisfy the requirements, all those are contained in $\boldsymbol{T}$ even though some of them are redundant because they are ``the same sequence with different subindexes''.
\end{definition}

In section 3, we studied the sequence $\boldsymbol{\Pi}$, which is periodic on the left side: the period is $1$ and its unitary sequence is the element $-2$. It is possible to \textit{construct} more sequences of this nature with the periodic ultra-recursive sequences we just found.
\begin{definition}
    For any sequence $(a_k)$ with period $p$, we denote as $(\breve{a}_z)^{\beta}_{z=-\infty}$ to the infinite subsequence of $(a_k)$ whose last element is $\breve{a}_\beta=a_0=a_p$. Here, $\beta$ can be any number, depending on the context.
    \begin{equation*}
        (\breve{a}_z)^\beta_{z=-\infty}\colon \ \breve{a}_{\beta-n}=a_{-n} \ \forall \ n\in\mathbb{N}
    \end{equation*}
\end{definition}

Let's consider the sequence $(\breve{\tau}^{5,1}_{1,z})$:
\begin{equation*}
    (\breve{\tau}^{5,1}_{1,z})=(...,-2,-2,6,-2,-6,-2,-2,-2,6,-2)
\end{equation*}
Notice that it is partially an ultrarrecursive sequence because every term generates the next by (2.3) except for the last element who does not have a successor. Since this element is $-2$, we can \textit{propose} any positive value as we did with $\boldsymbol{\Pi}$. The calculation of the first elements of first sequences is showed below:
\setcounter{MaxMatrixCols}{20}
\begin{equation*}
    \begin{pmatrix}
    (\breve{\tau}^{5,1}_{1,z})& 1 & 2 & 5 & 17 & 24 & 47 &93&174&321& \ldots\\
    (\breve{\tau}^{5,1}_{1,z})& 2 & 2 & 6 & 18 & 34 & 62 &118&218&398& \ldots\\
    (\breve{\tau}^{5,1}_{1,z})& 3 & 10 & 19 & 35 & 60 & 113  &215&398&731& \ldots\\
    (\breve{\tau}^{5,1}_{1,z})& 4 & 10 & 20 & 36 & 70 & 128 &240&442&820& \ldots\\
    (\breve{\tau}^{5,1}_{1,z})& 5 & 10 & 21 & 33 & 68 & 127 &229&426&793& \ldots\\
    (\breve{\tau}^{5,1}_{1,z})& 6 & 10 & 22 & 34 & 66 & 122 &234&430&798& \ldots\\
    \vdots & \vdots &\vdots & \vdots & \vdots & \vdots & \vdots & \vdots & \vdots & \vdots &\ddots
    \end{pmatrix}
\end{equation*}
The first thing that catches the eye is that, not like in $\boldsymbol{\Pi}$ nor $\boldsymbol{\Pi}^*$, in some cases, the properties $\tau_{m,n}<\tau_{m+1,n}$ and $\Delta{(\tau_{m,n})}<\Delta{(\tau_{m,n+1})}$ are not satisfied.\\
As we shall see later, the chaotic behavior of this sequence of sequences can represent an application in cryptography.
Before we aim to give an approximation of the $n$-th element of such a sequence, it is necessary to introduce the following theorem.
\begin{theorem}
   For any sequence $(a_k)$ with two consecutive elements that satisfy $0<a_n<a_{n+1}$ it's true that 
   \begin{equation}
   \begin{split}
      &\sum^{n-a_n}_{i=n+2-a_{n+1}}a_i=(a_{n+1}-a_n-1)(-2)+R\\
   \end{split}
   \end{equation}
   for some $R$, which implies $\sum^{n-a_n}_{i=n+2-a_{n+1}}(a_i+2)=R$. Therefore, the sequence $(a^\prime_k)\equiv\textsc{O}\circ(a_k)$, has the element $a^\prime_{n+2}$:
   \begin{equation}
       a^\prime_{n+2}=a^\prime_{n+1}+a_n+R+2
   \end{equation}
\end{theorem}
\begin{proof}
By equation (2.6):
\begin{equation*}
\begin{split}
    a^\prime_{n+2}&=2+ \sum^{n}_{i=n+2-a_{n+1}}(a_i+2)\\
    &=2+a_n+2+\sum^{n-1}_{i=n+1-a_n}(a_i+2)+\sum^{n-a_n}_{i=n+2-a_{n+1}}(a_i+2)\\
    &=2+a_n+a^\prime_{n+1}+R
\end{split}
\end{equation*}
\end{proof}
\begin{corollary}
For any ultrarecursive sequence $(u_k)$, if $0<a_n<a_{n+1}$, it's true that
\begin{equation}
    a_{n+2}=a_{n+1}+a_n+R+2
\end{equation}
with $R=\sum^{n-a_n}_{i=n+2-a_{n+1}}(a_i+2)$.
\end{corollary}
\begin{corollary}
    In $(\pi_{m,n})$, every term $\pi_{m,n}$ is greater than $n$. Therefore:
    \begin{equation*}
        \pi_{m,n+2}=\pi_{m,n+2}+\pi_{m,n}+2
    \end{equation*}
\end{corollary}
\begin{proof}
    Since $\pi_{m,k}=-2$ for $k<0$, by Corollary 4.1:
    \begin{equation*}
        R=\sum^{n-\pi_{m,n}}_{i=n+2-\pi_{m,n+1}}(\pi_{m,i}+2)=\sum^{n-\pi_{m,n}}_{i=n+2-\pi_{m,n+1}}(0)=0
    \end{equation*}
\end{proof}
Equation (4.1) can be explained as follows: for any sequence $(a_k)$, it's possible to interpret any element $a_n$ as $-2+r_n$. Therefore, the sum of consecutive elements of the sequence is:
\begin{equation*}
\begin{split}
    \sum^\beta_{i=\alpha}a_i=\sum^\beta_{i=\alpha}(-2+r_i)&=(\beta+1-\alpha)\Biggl{(}-2+\frac{1}{\beta+1-\alpha}\sum^\beta_{i=\alpha}r_i\Biggr{)}\\
    &=(\beta+1-\alpha)(-2+\bar{r})
\end{split}
\end{equation*}
where $\bar{r}$ is the average of that set of consecutive elements and in (4.1) $R=(a_{n+1}-a_n-1)\bar{r}$. In the following theorem, we'll assume that the average of a large set of consecutive elements of a periodic sequence is near to the average of the unitary sequence:
\begin{equation}
    \beta-\alpha\ggg1\implies\bar{r}\equiv\frac{1}{\beta+1-\alpha}\sum^\beta_{i=\alpha}r_i\thickapprox\frac{1}{p}\sum^p_{i=1}\dot{r}_i
\end{equation}

\begin{theorem}
   For every ultrarecursive sequence $(u_k)$ periodic through the left with $(\breve{\tau}^{\textsc{P},\textsc{N}}_{m,z})^\alpha_{z=-\infty}$. Given two elements $0<n-\alpha\lll u_n<u_{n+1}$, the next element of the sequence is:
   \begin{equation}
       u_{n+2}\thickapprox u_{n+1}(\xi_m)+u_n(2-\xi_m)+3-\xi_m
   \end{equation}
   with $\xi_m=2-\frac{1}{2m+1}$. Therefore, the approximate solution for any $u_{n+r}$ would be:
   \begin{equation}
       u_{n+r}=\kappa^+_{m,n} \phi^r_m + \kappa^-_{m,n} (\xi_m-\phi_m)^r
   \end{equation}
   where $\phi_m=\frac{1}{2}\bigl{(}\xi_m+\sqrt{(\xi_m-2)^2+4}\bigr{)}$ and the $\kappa$ constants:
   \begin{equation*}
    \begin{pmatrix}
    \kappa^+_{m,n}\\\kappa^-_{m,n}
    \end{pmatrix}=\frac{1}{\sqrt{(\xi_m-2)^2+4}}\begin{pmatrix}
    (\xi_m-\phi_m)u_n-u_{n+1}\\u_{n+1}-\phi_m u_n
    \end{pmatrix}
\end{equation*}
\end{theorem}
\begin{proof}
First, we find the sum of all the residues $r_n$ in $(\dot{\tau}^{\textsc{P},\textsc{N}}_{m,n})$:
\begin{equation*}
\begin{split}
    \sum_{p\in\textsc{P}}\dot{\tau}^{\textsc{P},\textsc{N}}_{m,p}+\sum_{n\in\textsc{N}}\dot{\tau}^{\textsc{P},\textsc{N}}_{m,n}+\sum_{q\notin\textsc{Q}}\dot{\tau}^{\textsc{P},\textsc{N}}_{m,q}=\sum_{p\in\textsc{P}}(4m+2)+\sum_{n\in\textsc{N}}(-4m-2)+\sum_{q\notin\textsc{Q}}(-2)\\
    =\sum_{p\in\textsc{P}}(-2+r_p)+\sum_{n\in\textsc{N}}(-2+r_n)+\sum_{q\notin\textsc{Q}}(-2+r_q)=-4m-4\\
    \implies\sum{r_i}=(-4m-4)+(2m)+(2m)+(4m+4)=4m \quad \quad\\
\end{split}
\end{equation*}
Therefore, the average $\bar{r}_m$ for the $4m+2$ elements of the unitary sequence is $\bar{r}_m=\frac{2m}{2m+1}$. By equation (4.3):
\begin{equation*}
    u_{n+2}\thickapprox u_{n+1}+u_n+(u_{n+1}-u_n-1)\Biggl{(}\frac{2m}{2m+1}\Biggr{)}+2
\end{equation*}
This is (4.5) if $\xi_m\equiv 1+\bar{r}_m=1+\frac{2m}{2m+1}$. The following step is to find the closed-form expression of $(u_k)$ for every $m$.
\end{proof}
Lets consider the sequence $(\breve{\tau}^{\textsc{P},\textsc{N}}_{2,k})$ with $\textsc{P}=\{6,9\}$ and $\textsc{N}=\{1,3\}$:
\begin{equation*}
    (...,-10,-2,\textbf{-10},-2,-2,10,-2,-2,10,-2)
\end{equation*}
Again, this is partially an ultrarecursive sequence because almost every term generate the next by the definitions, except for the minus $10$ written in bold text and the last term (which does not have any successor to generate). For $\textbf{-10}$ it is not possible to generate the next term because there are not enough elements in the sequence: two elements to the right are needed and its sum must be $-12$. For these reasons, it seems impossible to expand the sequence with a positive element as in $\boldsymbol{\Pi}$, the unitary sequences of $\boldsymbol{T}$ for $m=2$ are the only u-recursive sequences we know that have two consecutive elements whose sum is $-12$. Of course, we can combine elements $\ldots(\dot{\tau}^{\textsc{P}_i,\textsc{N}_i}_{2,k})(\dot{\tau}^{\textsc{P}_{i+1},\textsc{N}_{i+1}}_{2,k})(\dot{\tau}^{\textsc{P}_{i+2},\textsc{N}_{i+2}}_{2,k})\ldots$ but this is kind of boring, it doesn't worth further analysis and does not give us more valuable information about the properties of the ultra-recursive sequences. Moreover, we need to find those sequences Theorem 4.3 talks about.
\begin{definition}
    We say that $(a_n)^{\beta}_{n=\alpha}$ is a \textit{free} ultra-recursive sequence (or just free u-recursive sequence) if it satisfies the following three conditions: 
    \begin{enumerate}
        \item $\alpha\leq n+\sign{a_n}-a_n\leq\beta$ for $\alpha\leq n<\beta$.
        \item Every term $a_n$ for $n<\beta$, generates its successor by equation (2.2).
        \item $a_\beta=-2$.
    \end{enumerate}
\end{definition}

A notable example of a free u-recursive sequence is $(-2_n)^{0}_{n=\alpha}$ for $\alpha\in(-\infty,0]$.\
The sub-indexes here are irrelevant, we just care about the size of the free u-recursive sequence, which can be infinite.

\begin{corollary}
If $(a_n)$ and $(b_n)$ are free u-recursive sequences, then $((a_n),(b_n))$ is also free if $(b_n)$ is not infinite.
\end{corollary}

\begin{theorem}
   Exists an ultrarecursive sequence $(\omega_k)$ whose elements different from $-2$ are $\omega_{(2n+1)}=-\omega_{-2n}=(4n+2) \ \ \forall \ \ n>0$
   \begin{equation*}
       (...,-14,-2,-10,-2,-6,-2,-2,-2,-2,6,-2,10,-2,14,...)
   \end{equation*}
   And every subsequence $(\omega_k)^{2n+2}_{i=-2n}$ is free.
\end{theorem}
\begin{proof}
Suppose that for some $4n+2$, there exists a free u-recursive sequence $(u_m)$ of $4n-1$ elements such that $\sum^{4n-1}_{i=1}\dot{u}_i=-4n-2$.\\
For $n=2$, $(-6,-2,-2,-2,-2,6,-2)$ is a free sequence with those properties. The proof is complete by induction if we prove that 
\begin{equation}
    (u^*_m)=(-4n-2,-2,(u_m),4n+2,-2)
\end{equation} 
is also free, and it has $4(n+1)-1$ elements such that $\sum^{4(n+1)-1}_{i=1}\dot{u}^*_i=-4(n+1)-2$. 
First, we prove that $\dot{u}^*_1=-4n-2$ generates $\dot{u}^*_2=-2$ according to Corollary 2.2:
\begin{equation*}
    \sum^{4n+2}_{i=3}\dot{u}^*_i=\sum^{4n-1}_{i=1}\dot{u}_i+(4n+2)=(-4n-2)+(4n+2)=0
\end{equation*}
We can use this same result to prove that $\dot{u}^*_{4n+2}=4n+2$ generates $\dot{u}^*_{4n+3}=-2$ by equation (2.5)
\begin{equation*}
\begin{split}
    \dot{u}^*_{4n+2}+\sum^{4n+2}_{i=1}\dot{u}^*_i&=(4n+2)+\dot{u}^*_1+\dot{u}^*_2+\sum^{4n+2}_{i=3}\dot{u}^*_i\\
    &=(4n+2)+(-4n-2)+(-2)+0=-2
\end{split}
\end{equation*}
\end{proof}
Note that $(\omega_n)^6_{n=-4}=(-10,-2,-6,-2,-2,-2,-2,6,-2,10,-2)$ has the properties we were looking for before: the sum of its first two elements is $-12$. Now, we can generate a whole new group of sequences with $((\breve{\tau}^{\textsc{P},\textsc{N}}_{2,k}),(\omega_n)^6_{n=-4},m)$ where $\textsc{P}=\{6,9\}$, $\textsc{N}=\{1,3\}$ and $m>0$. For $m=1$, the following sequence takes place: 
\begin{equation*}
    ((\breve{\tau}^{\textsc{P},\textsc{N}}_{2,k}), (\omega_n)^6_{n=-4},1 , 2 , 5 , 21 , 48 , 83 ,169,302,589,1121,2128 ,4075,7753,\ldots)
\end{equation*}

Finally, lets consider the infinite subsequence $(\breve{\tau}^{\textsc{P},\textsc{N}}_{3,k})$ for $\textsc{P}=\{8,11,13\}$ and $\textsc{N}=\{1,3,6\}$:
\begin{equation*}
    (\ldots,-14,-2,\boldsymbol{-14},-2,-2,\boldsymbol{-14},-2,14,-2,-2,14,-2,14,-2)
\end{equation*}
Again, this is not a free u-recursive sequence because the two bold $-14$ can't generate its successors: there are not enough elements in the sequence. It is needed to add to the right a finit free u-recursive sequence $(a_n)$ so that $\dot{a}_1+\dot{a}_2=-16$ and $\sum^5_{i=1}\dot{a}_i=-34$. Fortunately, we can use again a free subsequence of $(\omega_k)$, since the first terms of $(\omega_n)^8_{n=-6}$ are $(-14,-2,-10,-2,-6,\ldots)$. Lets compute the following terms of the sequence if we also add the element $1$.
\begin{equation*}
    ((\breve{\tau}^{\textsc{P},\textsc{N}}_{3,k}),(\omega_n)^8_{n=-6},1,2,5,25,60, 103, 201, 402, 749, 1477, 2852, 5495, 10641,\ldots)
\end{equation*}

\section{Periodic eigen-sequences of the transformation $\textsc{O}^n$}
Several examples of eigen-sequences of the transformation $\textsc{O}$ have been found. In this section we discuss briefly the existence of sequences $(A_k)$ that remain invariant only after applying $r(>1)$ times $\textsc{O}$; the easiest way of start looking for such sequences  is to imagining that $\textsc{O}$ has the same effect in $(A_k)$ as the following transformation:
\begin{equation}
    \textsc{L}\circ(A_k)\equiv(A^\prime_k)\colon \ A^\prime_{p+1}=A_p
\end{equation}
if $(A_k)$ is periodic with period $r$, it is clear that $\textsc{L}^r\circ(A_k)=(A_k)$. Combining (5.1) and (2.3) leads to the following equation for $|A_p|>0$:
\begin{equation}
    A^\prime_{p+1}=A_p=|A_p|+\sum^{|A_p|-1}_{i=0}A_{p-i\sign{A_p}}
\end{equation}
This can be interpreted as if every term generates itself instead of its next term. This is always true for $A_p=0$. Equation (5.2) is true if and only if
\begin{equation}
    |A_p|=-\sum^{|A_p|-1}_{i=1}A_{p-i\sign{A_p}}
\end{equation}

Since this sequences has period $r$, equation (5.3) is equivalent to:
\begin{equation}
    |A_p|=-n\sum^r_{i=1}\dot{A}_i-\sum^{|A_p|-1 \bmod{(r)}}_{i=1}A_{p-i\sign{A_p}}
\end{equation}
for some integer $n$.

\begin{theorem}
Given a sequence $(A_k)$ with period $r$ such that $|A_p|\neq0\implies |A_p|=r+1$ and $\sum^{r}_{i=1}\dot{A}_i=-(r+1)$, it is an eigen-sequence of $\textsc{O}^r$.
\end{theorem}
\begin{proof}
Equation (5.4) is satisfied under the conditions stated above:
\begin{equation*}
\begin{split}
    |A_p|&=|r+1|=r+1=-1*(-r-1)+0\\
    &=-1*\sum^r_{i=1}{\dot{A}_i}-\sum^{r \bmod{r}}_{i=1}A_{p-i\sign{A_p}}\\
\end{split}
\end{equation*}
\end{proof}
\begin{corollary}
For every $m>0$ there exists an eigen-sequence of $\textsc{O}^{2m+1}$ such that $m$ elements have value $2m+2$ and $m+1$ elements have value $-2m-2$.
\end{corollary}


\Addresses
\end{document}